%% file: 2024-02-29_MRMR_v3.tex
\documentclass[10pt,reqno]{amsart}

\usepackage[a4paper,left=35mm,right=35mm,top=30mm,bottom=30mm,marginpar=25mm]{geometry}
\usepackage[backend=bibtex,style=numeric]{biblatex}
\usepackage{amsxtra}
\usepackage{amsmath}
\usepackage{amsfonts,latexsym,amstext}
\usepackage{amssymb}
\usepackage{amsthm}
\usepackage{esint}
\usepackage{eurosym}
\usepackage[dvips]{graphics}
\usepackage{graphicx}
\usepackage{epsfig}
\usepackage{mathtools}
\usepackage{enumerate}
\usepackage{verbatim}
\usepackage{hyperref}
\usepackage{dsfont}
\usepackage{bbm}
\usepackage{dsfont}
\usepackage{bm}
\usepackage{mathrsfs}
\usepackage{relsize}
\usepackage{textcomp}
\usepackage[leqno]{amsmath}
\usepackage{color}
\input mymacro
\usepackage[displaymath,mathlines]{lineno}
\allowdisplaybreaks
\usepackage{ifthen}
\makeatletter
\makeindex
\makeatother

\usepackage{xcolor}

\numberwithin{equation}{section}

\newtheoremstyle{thmlemcorr}{10pt}{10pt}{\itshape}{}{\bfseries}{.}{10pt}{{\thmname{#1}\thmnumber{
#2}\thmnote{ (#3)}}}
\newtheoremstyle{teolemcorr*}{10pt}{10pt}{\itshape}{}{\bfseries}{.}\newline{{\thmname{#1}\thmnumber{
#2}\thmnote{ (#3)}}}
\newtheoremstyle{defi}{10pt}{10pt}{\itshape}{}{\bfseries}{.}{10pt}{{\thmname{#1}\thmnumber{
#2}\thmnote{ (#3)}}}
\newtheoremstyle{remexample}{10pt}{10pt}{}{}{\bfseries}{.}{10pt}{{\thmname{#1}\thmnumber{
#2}\thmnote{ (#3)}}}
\newtheoremstyle{ass}{10pt}{10pt}{}{}{\bfseries}{.}{10pt}{{\thmname{#1}\thmnumber{
A#2}\thmnote{ (#3)}}}

\theoremstyle{thmlemcorr}
\newtheorem{theorem}{Theorem}
\numberwithin{theorem}{section}

\newtheorem{pro}[theorem]{Proposition}

\theoremstyle{thmlemcorr*}
\newtheorem{theorem*}{Theorem}
\newtheorem{lemma*}[theorem]{Lemma}
\newtheorem{corollary*}[theorem]{Corollary}
\newtheorem{proposition*}[theorem]{Proposition}
\newtheorem{problem*}[theorem]{Problem}
\newtheorem{conjecture*}[theorem]{Conjecture}

\theoremstyle{defi}
\newtheorem{definition}[theorem]{Definition}
\newtheorem{hyp}{Assumptions}

\theoremstyle{remexample}

\theoremstyle{ass}

\begin{document}

\title[MFG in Optimal Advertising Models]{Mean Field Games Incorporating Carryover Effects: Optimizing Advertising Models}

\author{Michele Ricciardi}
\address[M. Ricciardi]{
Università LUISS - Guido Carli. Viale Romania 32, 00197 Roma, Italy}
\email{mricciardi@luiss.it}
\author{Mauro Rosestolato}
\address[M. Rosestolato]{
Università di Genova. Via Vivaldi 5, 12126 Genova, Italy}
\email{mauro.rosestolato@unige.it}

%\keywords{Mean Field Games; Neumann conditions; }
%\subjclass[2010]{
%35J47, %Second order elliptic systems
%35A01} %Existence problems: global existence, local existence, non-existence
%
%\thanks{
%}
\date{}

\begin{abstract}
	We consider a class of optimal control problems that arise in connection with optimal advertising under uncertainty.
	Two main features appear in the model: a delay in the control variable driving the state dynamics; a mean-field term both in the state dynamics and in the utility functional, taking into account for other agents.
	We interpret the model in a competitive environment, hence we set it in the framework of Mean Field Games.
	We rephrase the problem in an infinite dimensional setting, in order to obtain the associated Mean Field Game system.  Finally, we
	specify the problem to a simple case, and solve it providing an explicit solution.
\end{abstract}

\maketitle
%\tableofcontents
\section{Introduction}

Optimal advertising problems have been  largely studied, starting from the  seminal papers
\cite{VidaleWolfe1957,NerloveArrow1962}, through various extensions, including
the case in a stochastic environment
(\cite{PrasadSethi2004,ViscolaniGrosset2004,Marinelli2007,motte2021optimal}), and including a delay effect on the control variable
(\cite{GozziMarinelli,GozziMarinelliSavin,Hartl1984,Feichtingeretal1994}).

Game theory is the natural theoretical framework to be considered
when the outcome of each agent
is influenced by the actions taken by other agents, as is the case
with advertising problems.
But when the number of agents is large, the game theory approach is difficult to exploit.
In such a case, when a large number of agents occurs,
other approaches are possible, under the assumption that 
the agents  maximize the same functional by controlling the same dynamics.
Indeed, by considering a
 limit problem formally
associated to the original game,
one can obtain a new class of problems whose structure is still representative and meaningful for the intended application (see \cite{carmona2018probabilistic}).
In the limit problem, the fact that the state equation of each agent is influenced by all the  others is expressed through a mean-field term depending on the law of the state variable, which is the key feature of such class of problems.

There are two ways to interpret the mean-field term.
The agent optimizing her utility within this mean-field framework can consider
the mean-field term
as endogenous, meaning,
as determined by her own state distribution.
A different perspective occurs when the agent   considers
the mean-field term as a given exogenous term, representing the influence of the  distribution  of the aggregation of
all the other agents' states.

The first perspective is suitable for applications to non-competitive environments, and the
theoretical framework is that of McKean-Vlasov optimal control
(see the pioneering work \cite{Kac1956FoundationsOK,MR0233437} and the seminal paper \cite{Sznitman1991TopicsIP} in the framework of propagation of chaos; more recent contributions are, e.g., \cite{Pham2017,COSSO2019180,Djete2019McKeanVlasovOC}; see also \cite{wu2019viscosity} for an extension to the path-dependent case, and \cite{Cosso20232863,Cosso202431} for the path-dependent case in infinite dimension).
The second perspective is suitable for applications to competitive environments, and the machinery to deal with is that of Mean Field Games (MFG), a theory introduced in 2006 in \cite{ll1,ll2,ll3}, using tools from mean-field theories, and in the same years in \cite{Caines1}. See also, among the extended literature on this topic, \cite{bensoussan,cardaliaguet}).

% We briefly sketch the difference in the theoretical framework between the  McKean-Vlasov optimal control and the MFG approaches.

In the MFG approach, the key assumption is that at the Nash equilibria the distribution of the aggregate state of the other agents (agents are homogeneous, with the same utility and state dynamics) is assumed to be
equal to the distribution of the state of the single agent optimizing her own position by assuming the aggregate one as given.
This assumption leads to the MFG system, where the backward Hamilton-Jacobi-Bellman  (HJB) equation associated to the
optimal control of the single agent (where the mean-field term appears in the Hamiltonian and/or in the terminal condition) is coupled with the forward Fokker-Planck equation for the
distribution of the optimal state.

In this paper we address a competitive framework, hence we fall into the MFG class of problems.
But, for the presence of the mean-field term, we add a second feature in the state dynamics: a delay effect in the control.
Indeed, the state dynamics that we consider has the form
\begin{equation*}
dX(s)=\left[aX(s)+\gamma(\tilde m_0(s))+b_0u(s)+\intz b_1(\xi)u(s+\xi)\,d\xi\right]ds+\sigma dW(s),
% X(t)=x\,,\\
% u(t+\xi)=\delta(\xi)\,,\qquad\forall\,\xi\in[-d,0]\,.
\end{equation*}
where the delay effect in the control is due to the integral term, and the mean-field term $\gamma(\tilde{m}_0)$ appears, with $\tilde{m}_0$ representing the given distribution of the aggregate states.

When considering a delay effect in the control variable, the standard dynamic programming approach to  derive the HJB equation cannot be applied to
the value function of the agent, due to the fact that the state dynamics is not Markovian.
A useful and well-exploited machinery  to overcome this impasse consists in rephrasing
the finite dimensional state with delay in the control in an infinite dimensional setting where no delay appears.
% ({\color{red}CITARE}).
The infinite dimensional representation does not lack Markovianity anymore, hence the dynamic programming approach can be applied, though the theory to be used must now be available for infinite dimensional spaces (the main reference here is
\cite{GozziBOOK}).

By exploiting the two theoretical frameworks -- MFG and reformulation in infinite dimension -- we end up with a MFG system where the HJB equation is set on an infinite dimensional space.

We stress the fact that, although the Mean Field Games  with finite-dimensional state space have been widely studied, very few results are available in the infinite-dimensional case. To our knowledge, the only attempt in this direction is given by \cite{federico2024linearquadratic, FouqueZhang}.

However, our aim here is not to develop a general theory for infinite dimensional MFG systems. Indeed, after representing the problem in a suitably chosen Hilbert space, we specify our data to the linear case and search for an explicit solution, that we can find also because the MFG system becomes uncoupled.
We refer the reader to
\cite{gozzi2024optimal}
for an analogous  problem, but formulated in a non-competitive framework formally set by a McKean-Vlasov optimal control.

%The same problem, but within the McKean-Vlasov optimal control framework, has been treated in {\color{red}CITA}.

\bigskip
The paper is organized as follows.
In 
Section~\ref{sec_problem}
we formulate the optimal advertising problem
as an optimal control problem for a stochastic
differential equations
with  mean field term and delay in the control.
We then introduce the infinite dimensional setting in which we rephrase the original problem with delay as
an infinite dimensional optimal control problem without delay.
In Section~\ref{sec_mfgsystem}
the Mean Field Game system is derived.
Finally, in Section~\ref{sec_solution}
we specify the model in a particular case and we find an explicit solution.

\section{The problem in its general form}\label{sec_problem}
Let $t\in[0,T]$. We consider the following stochastic delay differential equation (SDDE):
\begin{equation}\label{SDDE}
\begin{cases}
dX(s)=\left[aX(s)+\gamma(\tilde m_0(s))+b_0u(s)+\intz b_1(\xi)u(s+\xi)\,d\xi\right]ds+\sigma dW(s)\,,\\
X(t)=x\,,\\
u(t+\xi)=\delta(\xi)\,,\qquad\forall\,\xi\in[-d,0]\,.
\end{cases}
\end{equation}

\noindent Here, the solution $X^{t,x,u,\tilde m_0}$ (from now on we omit for simplicity the dependence on $\tilde m_0$)
denotes the stock of advertising goodwill, with initial quantity $x$ at time $t$.
%To avoid too heavy notation, we simply write $X(\cdot)$ when there is no possibility of mistake.
The Brownian motion $W$ is defined on a
filtered probability space $(\Omega,\mathcal F, (\mathcal F_t)_{t\ge0},\P)$,
with $(\Omega,\mathcal{F},\mathbb{P})$ being complete,  and $\mathbb{F}$ being the augmentation of the filtration generated by $W$.
The function $\tilde m_0$ represents the mass distribution of the population at time $s$,
% (the notation $\tilde m_0$ instead of the classical $m$ will be clearer later),
whereas $u$ is the control strategy,
valued in $\mathbb{R}$.
% a given closed interval $U\subset \mathbb{R}$. 
We require
 $$
 u\in\mathcal U\coloneqq L^2_{\mathcal P}(\Omega\times[0,T];\mathbb{R}),
 $$
 the space of square integrable
 real-valued
 progressively measurable processes.

Notice that the integral term
appearing in the dynamics of $X$
takes into account also the past of $u$, where the control $u$ is extended on $[-d,0]$
by a given (deterministic) function $\delta$, representing
the advertising expenditure in the past up to time $0$.
%We assume $\delta\in L^2([-d,0]; U)$.

 In \eqref{SDDE}, we also  take $a,b_0\in\R$, $b_1\colon [-d,0]\to\R$, and $\gamma\colon\mathcal P_1(\R)\to\R$, where $\mathcal P_1(\R)$ denotes the set of Borel probability measures on $\R$ with finite first-order moment, meaning
 \begin{equation*}
   \mathcal{P}_1(\mathbb{R})
   \coloneqq
   \left\{
     \mathbb{\mu}
     \textrm{ probability measure on $(\mathbb{R},\mathcal{B}(\mathbb{R}))$}
     \colon
     \int_{\mathbb{R}}|x|\mu(dx)<\infty.
   \right\} 
 \end{equation*}
 The set $\mathcal{P}_1(\mathbb{R})$ is endowed with the
 Monge-Kantorovitch distance $\mathbf{d}_1$ (see e.g.\ \cite{MR4214773}).

 \smallskip
 \noindent The objective functional is defined in the following way:
 % {\color{magenta}(indicare anche la dipendenza da $\tilde{m}_0$?)} {\color{blue} Di solito non si fa in ambito MFG}
\begin{equation}\label{J}
  \begin{split}
    &J(t,x,u)=\E\left[\int_t^T e^{-r(s-t)}\Big[f(s,X^{t,x,u}(s),\tilde m_0(s))-L(s,X^{t,x,u}(s),u(s))\Big]ds\right.\\
&\left.\phantom{\int_0^T}
   \hspace{6.22cm}
  +e^{-r(T-t)}g(X^{t,x,u}(T),\tilde m_0(T))\right],
\end{split}
\end{equation}
where $r\in\R^+$, $f:[0,T]\times\R\times\mathcal P_1(\R)\to\R$, $L:[0,T]\times\R^2\to\R$, $g:\R\times\mathcal P_1(\R)\to\R$. Here, $L$ represents the Lagrangian cost for the advertising, whereas $f$ and $g$ are, respectively, the running utility and the terminal utility of the objective functional, which may also depend on the distribution function  $\tilde m_0$.

\smallskip
Throughout the paper, the following  Assumptions~\ref{hp} will be standing
% {\color{magenta}forse conviene mettere a parte la condizione di continuit\`a per $\tilde{m}_0$}. {\color{blue} In realtà alla fine si fa un punto fisso, e $\tilde m_0$ diventa soluzione di un HJB. Si può mettere una condizione di continuità, visto che alla fine sarà soddisfatta, ma non è fondamentale}

\begin{hyp}\label{hp}
\phantom{phantom}
\begin{enumerate}[a)]
\item $b_0\ge0$ and $r\ge0$, whereas $b_1\in L^2([-d,0];\R^+)$, $\delta\in L^2([-d,0];\mathbb{R})$;
\item 
  $\tilde{m}_0\colon [0,T]\rightarrow \mathcal{P}_1(\mathbb{R})$ is continuous;
\item $\gamma\colon\mathcal P_1(\Omega)\to\R$ is continuous;
  % {\color{magenta}(dove si usa questa ipotesi? perch\'e non basta la misurabilit\`a, pi\`u limitatezza?)} \textcolor{blue}{Potrebbero bastare misurabilità e limitatezza in effetti, possiamo riformulare così.}\\
  % {\color{magenta}A occhio non mi sembra serva una condizione di crescita, dato che comunque se $m$ \`e continua allora il suo range \`e in un compatto quindi $\gamma(\tilde{m}_0)$ \`e bounded, quindi nn ci sono problemi per la soluzione della SDE.} \textcolor{blue}{Se $\gamma$ è solo misurabile non è detto che $\gamma(\tilde m_0)$ sia limitato, anche con $\tilde m_0$ continuo}
\item $f$, $g$ and $L$ are measurable in all their variable,  and locally uniformly continuous in the variables $(x,m)$, uniformly with respect to $(t,u)$, meaning that for every $R>0$ there exists a modulus of continuity $\mathtt{w}_R\colon \mathbb{R}^+\rightarrow \mathbb{R}^+$ such that
  \begin{multline*}
    \sup_{\substack{t\in[0,T]\\ u\in \mathbb{R}}}
     \Big( 
    |f(t,x,u)-f(t,x',u)|
    +
    |g(x,m)-g(x',m')|
    +
    |L(t,x,u)-L(t,x',u)|
  \Big)\\
  \leq
  \mathtt{w}_R
  \left(
    |x-x'|+\mathbf{d}_1(m,m')
  \right) ,
\end{multline*}
for all $x,x'\in \mathbb{R}$ and $m,m'\in \mathcal{P}_1(\mathbb{R})$ such that $|x|\vee |x'|\vee \mathbf{d}_1(m,\delta_0)\vee \mathbf{d}_1(m',\delta_0)\leq R$;
\item
  % {\color{magenta}(aggiungerei questa condizione per la buona definizione di $J$)} {\color{blue} esempi tipici si hanno con $L(u)=\frac 1p|u|^{p}$, $p>1$, quindi richiedere $L$ sublineare è un po' troppo. Ma non serve, la condizione $f$ garantisce il fatto che $L$ non dia problemi. Piuttosto potrebbe servire avere $f$ e $g$ globalmente limitate, per stare tranquilli che il $\sup J$ non esploda. Forse si può scendere a $f$ e $g$ sublineare, ma bisognerebbe richiedere qualche ipotesi in più su $\gamma$. È un po' tardi per pensarci, forse potremmo richiedere qualche condizione in più su $\gamma$ e amen}
  for any $m\in \mathcal{P}_1(\mathbb{R})$,
  there exists $N_m>0$ such that
  \begin{equation*}
    |f(t,x,m)|+|g(x,m)|\leq
    N_m \left(
      1+|x|+
      % \mathbf{d}_1(m,\delta_0)+
      |u|
    \right),
  \end{equation*}
  for all $t\in[0,T], x\in \mathbb{R},  u\in \mathbb{R}$;
\item $L$ is strictly convex
  % {\color{magenta}(dove viene usata la convessit\`a?)} {\color{blue} Poco sopra \eqref{optctrl}, la stretta convessità di $u$ si usa per avere $L_u$ strettamente crescente, e per far sì che il controllo ottimale sia univocamente definito}
  and coercive in the control variable, with super-linear growth,
   i.e., there exist $\eps, C_0, C_1>0$ such that
$$
L(t,x,u)\ge C_1 u^{1+\eps}-C_0\,,\qquad\forall t\in[0,T],\,x\in\R,\,u\in \mathbb{R}.
$$
Moreover, $L$ is differentiable in the control variable.
%{\color{blue} Ho aggiunto questa ipotesi, ci serve dopo quando facciamo i conti con $L_u$}
\end{enumerate}
\end{hyp}

\smallskip
Under
Assumptions~\ref{hp}, it is easily seen
that,
for any
$t\in [0,T), x\in \mathbb{R}, u\in \mathcal{U}$, there exists a
unique solution
$X^{t,x,u}$ to
\eqref{SDDE}, that
$X^{t,x,u}\in L^p_\mathcal{P}(\Omega\times [0,T];\mathbb{R})$ for all $p\geq 1$.
Moreover,
Assumptions~\ref{hp} guarantees that
the reward functional $J$ in
\eqref{J}
is well-defined and finite for any $(t,x,u)\in [0,T]\times
\mathbb{R}\times \mathcal{U}$.

\smallskip
The players choose the best control to maximize the objective functional $J$. To this aim, we define the value function as follows:
\begin{equation}\label{origV}
  V(t,x)=\sup_{u\in\mathcal U}J(t,x,u)\qquad \forall (t,x)\in [0,T]\times \mathbb{R}.
\end{equation}
Assumptions~\ref{hp} assures that
$V$ is finite.

\smallskip
Usually, the Hamilton-Jacobi equation satisfied by the value function is computed by applying the dynamic programming principle and Ito's formula to
 $V$. In this case, the integral term appearing in \eqref{SDDE} makes the process $X^{t,x,u}$ non-Markovian, hence the application of the dynamic programming principle is not at all obvious.

 Then the approach we follow here
 is  to consider the process $X^{t,x,u}$ as the projection, on the first component, of a Markovian process $Y$ defined on a suitably chosen Hilbert space, and to write the Hamilton-Jacobi equation for the value function related to this process. More precisely, we rephrase in our setting the approach used in  \cite{GozziMarinelli,GozziMasiero,GozziMasieroRosestolato,VinterKwong}.

\subsection{Reformulation of the problem}\label{sec_reformulation}
We consider the Hilbert space $H$ defined by
$$
H\coloneqq \R\times L^2([-d,0];\R),
$$
endowed with the following scalar product and norm: for $\bo x=\big(x_0,x_1\big)\in H$, $\bo y=\big(y_0,y_1\big)\in H$,
$$
\langle \bo x,\bo y\rangle\coloneqq x_0y_0+\intz x_1(\xi)y_1(\xi)\,d\xi,\qquad \norm{\bo x}=\sqrt{x_0^2+\intz x_1^2(\xi)\,d\xi}.
$$
% From now on, we will simply use $\langle\cdot,\cdot\rangle$ to denote the scalar product in $H$.
Now consider the operator $A\colon \mathcal{D}(A)\rightarrow H$
defined by
% and the linear functional $A:\mathcal D(A)\subset H\to H$ as
\begin{equation}\label{defA}
A\bo x= \big(ax_0+x_1(0),-x'_1\big),
\end{equation}
where
$$
\mathcal D(A)=\Big\{\bo x\in H\,|\,\, x_1\in W^{1,2}([-d,0])\,,\,\, x_1(-d)=0 \}.
$$
The adjoint
$A^*\colon \mathcal D(A^*)\to H$ of $A$
is given by
$$
A^*\bo x=\big(ax_0,x'_1\big),
$$
% From the definition of adjoint one has, for $\bo x\in H$ with $x_1\in W^{1,2}([-d,0])$ and $\bo y\in\mathcal D(A)$,
% \begin{align*}
% \langle A^*\bo x,\bo y\rangle=\langle \bo x,A\bo y\rangle&=ax_0y_0+x_0y_1(0)-\intz x_1(\xi)y'_1(\xi)\,d\xi
% \\&=ax_0y_0+y_1(0)(x_0-x_1(0))+\intz x'_1(\xi)y_1(\xi)\,d\xi\\&=\Big\langle \big(ax_0,x'_1(\cdot)\big),\bo y\Big\rangle+y_1(0)(x_0-x_1(0))\,.
% \end{align*}
where
$$
\mathcal D(A^*)=\Big\{\bo x\in H\,|\,\, x_1\in W^{1,2}([-d,0])\,,\,\, x_1(0)=x_0 \}.
$$
The operators $A$ and $A^*$ generate two $C_0$-semigroups $\left\{e^{tA}\right\}_{t\in\R^+}$ and $\left\{e^{tA^*}\right\}_{t\in\R^+}$ on $H$, explicitly given by
\begin{align}
&e^{tA}\bo x=\left(e^{at}x_0+\intz\indic_{[-t,0]}(s)e^{a(t+s)}x_1(s)\,ds,\,x_1(\cdot-t)\indic_{[-d+t,0]}(\cdot)\right)\,,\qquad\forall\,\bo x\in H\,,\notag\\
&e^{tA^*}\!\bo x=\left(e^{at}x_0,\,\indic_{[-t,0]}(\cdot)e^{a(\cdot+t)}x_0+x_1(\cdot+t)\indic_{[-d,-t]}(\cdot)\right)\,,\qquad\forall\,\bo x\in H\,.\label{semigruppi}
\end{align}
Now define the operators $\Sigma\colon \R\to H$, $\Gamma\colon \mathcal P_1(\Omega)\to H$ and $B\colon\R\to H$ by
$$
\Sigma x\coloneqq \big(\sigma x,0\big)\,,\qquad \Gamma(m)=\big(\gamma(m),0\big)\,,\qquad Bx=\big(b_0x,b_1(\cdot)x\big)=x\bo b\,,
$$
where $\bo b=(b_0,b_1)$.
A straightforward computation shows that the adjoint $B^*\colon H\to\R$ is
$$
B^*\bo x=\langle \bo b,\bo x\rangle=b_0x_0+\intz b_1(\xi)x_1(\xi)\,d\xi\qquad  \forall \bo x\in H.
$$
% Even if the results in this paper would become easier working directly with the representation $Bx=x\bo b$ and $B^*\bo x=\langle \bo b,\bo x\rangle$, we prefer to keep the general notation $Bx$ and $B^*\bo x$, which could be useful if one aims to work with different kind of problems.
Now, for $t\in [0,T), \bo x\in H, u\in\mathcal U$, we consider the following SDE:
\begin{equation}\label{SDE}
\begin{cases}
d\bo X(s)=\big[A\bo X(s)+\Gamma(\tilde m_0(s))+Bu(s)\big]ds+\Sigma dW(s)\qquad \forall s\in (t,T)\\[5pt]
\bo X(t)=\bo x\\[5pt]
u(t+\xi)=\delta(\xi)\qquad\forall\,\xi\in[-d,0].
\end{cases}
\end{equation}
% with $t\in[0,T)$ and $\bo x\in H$.
The abstract stochastic differential equation
\eqref{SDE}
admits
a unique mild solution $\bo X^{t,\bo x,u}$
(see e.g.\ \cite{DaPrato2014}), i.e., the pathwise continuous process in $L^2_\mathcal{P}(\Omega\times [0,T];H)$ given by the variation of constants formula:
\begin{equation*}
  \bo X^{t,\bo x,u}(s)=e^{(s-t)A}\bo x+\int_t^s e^{(s-r)A}\big[\Gamma(\tilde m_0(r))+Bu(r)\big]\,dr+\int_t^s e^{(s-r)A}\Sigma\,dW(r)\,,\quad \forall s\in[t,T]\,.
\end{equation*}
% We immediately observe that $\bo X(\cdot)$ is a $H$-valued Markovian process. Moreover, the following hold true.
% Observe that the equation cannot be solved in the classical sense, since to apply the operator $A$ to $\bo X(\cdot)$, we should have $\bo X(s)\in\mathcal D(A)$ for all $s\in[t,T]$, which is not true in general. Actually, we have $\bo X(t)=\bo x$, which can be outside $\mathcal D(A)$. Hence, we adopt the following definition of mild solution.
% \begin{definition}\label{defmild}
% Let $\bo X\in C([0,T];L^2_{\mathcal P_1}(\Omega;H))$. We say that $\bo X$ is a mild solution of \eqref{SDE} if the variation of constants formula hold:
% \end{definition}
%Observe that a classical solution of the process is also a mild solution. Moreover, the operators $\{e^{sA}\}_{s\in\R^+}$ are defined in the whole space $H$, which makes Definition \ref{defmild} more general.

The connection
between \eqref{SDDE} and \eqref{SDE}
is clarified by
the following proposition.
For any $H$-valued process $Z$, we denote by $Z_0$ and $Z_1$ its orthogonal projections into the orthogonal components $\mathbb{R}$ and $L^2([-d,0];\mathbb{R})$ of $H=\mathbb{R}\times L^2([-d,0];\mathbb{R})$, respectively.

%The following proposition (see also ) gives a connection 
\begin{pro}\label{prop_equiv}
  % Suppose Assumptions \ref{hp} are satisfied.
  Let $t\in[0,T), x\in \mathbb{R},u\in \mathcal{U}$.
  Let 
  %Let the $H$-valued process $\bo X^{t,x,}$ be a mild solution of \eqref{SDE}, with $\bo x=(x_0,x_1)\in H$ satisfying
  \begin{equation}\label{x1}
      x_1(\zeta)\coloneqq \int_{-d}^\zeta b_1(\xi)\delta(\xi-\zeta)\,d\xi\qquad\forall\zeta\in[-d,0]
    \end{equation}
      and $\bo x=(x,x_1)$.
 % Let $(\bo X^{t,\bo x,u}_0,X_1)$ be the two marginal components of the process $\bo X$.
  Then  $\bo X^{t,\bo x,u}_0=X^{t,x,u}$.
\end{pro}

\begin{proof}
  Apart from the  term $\gamma(\tilde{m}_0)$ in~\eqref{SDDE}, wich does not cause any trouble, the proof is  the same as in 
  \cite{GozziMarinelli,difeo}.
\end{proof}

\smallskip
We now introduce
an auxiliary control problem.
For $t\in[0,T], \bo x\in H, u\in \mathcal{U}$, define
%the objective functional associated to \eqref{SDE}. For any 
\begin{equation}\label{obj}
  \begin{split}
          &\mathcal J(t,\bo x,u)\coloneqq
  \E\left[\int_t^T e^{-r(s-t)}\Big[f(s,
    \bo X^{t,\bo x,u}_0(s),\tilde m_0(s))-L(s,\bo X^{t,\bo x,u}_0(s),u(s))\Big]\,ds\right.
  \\
  &\hspace{6.5cm}
  \left.\textcolor{white}{\int_t^T}+e^{-r(T-t)}g\big(\bo X^{t,\bo x,u}_0(T),\tilde m_0(T)\big)\right].
\end{split}
\end{equation}
% where, as before, $X_0(\cdot)$ denotes the first marginal of the process $\bo X(\cdot)$.
The value function for the auxiliary problem is 
\begin{equation}\label{value}
\mathcal{V}(t,\bo x)=\sup\limits_{u\in\mathcal U}\mathcal J(t,\bo x,u)\qquad \forall (t,\bo x)\in [0,T]\times H.
\end{equation}
Notice that the auxiliary problem \eqref{value} generalizes the original
problem~\eqref{origV}, in the sense that, thanks to
Proposition~\ref{prop_equiv}, we have
\begin{equation*}
  V(t,x)=\mathcal{V}(t,\bo x),
\end{equation*}
where $\bo x=(x,x_1)$ and $x_1$ defined by~\eqref{x1}.

\section{The Mean Field Game system}\label{sec_mfgsystem}

\subsection{The Hamilton-Jacobi equation}\label{sec_HJB}
Tha Hamilton-Jacobi-Bellmann
equation formally associated to $\mathcal{V}$
is (see \cite{GozziBOOK})
% defined before satisfies a Hamilton-Jacobi equation, which can be obtained applying the dynamic programming principle and Ito's formula to the function $v$.
%The equation obtained is the following one, for $(t,\bo x)\in[0,T]\times H$:
\begin{equation}\label{hjb}
  \hskip-0.5cm
  \begin{cases}
\ds v_t(t,\bo x)+\frac12\tr\big(\Sigma\Sigma^*v_{\bo x\bo x}(t,\bo x)\big)+\mathcal H\big(t,\bo x,\tilde m_0(t),v_{\bo x}(t,\bo x)\big)+f\big(t,x_0,\tilde m_0(t)\big)=r v(t,\bo x)\\[5pt]
v(T,\bo x)=g(x_0,\tilde m_0(T))
\end{cases}
\end{equation}
for $(t,\bo x)\in[0,T)\times H$.
% As always, we adopt the notation $\bo x=(x_0,x_1)$.
% Here, $\mathcal H$ denotes t
Denoting by $L^*$ the convex conjugate of $L$ with respect to the $u$ variable, meaning
\begin{equation}\label{Lstar}
  L^*(t,x_0,\bo q)=
  \sup\limits_{u\in \mathbb{R}}\Big(\langle u,\bo q\rangle-L(t,x_0,u)\Big)\qquad  \forall (t,x_0,\bo q)\in [0,T]\times \mathbb{R}\times H,
\end{equation}
the Hamiltonian function $\mathcal H$
% $\colon [0,T]\times H\times\mathcal P_1(\R)\times \mathcal{D}(A)\to\R$
is defined by
\begin{equation}\label{supremo}
  \mathcal H(t,\bo x,m,\bo p)=\langle A\bo x
  +\Gamma(m),\bo p\rangle+
  L^*(t,x_0,B^*\bo p)
  %\sup\limits_{u\in U}\Big(\langle Bu,\bo p\rangle-L(t,x_0,u)\Big),
\end{equation}
for $t\in [0,T), \bo x\in \mathcal{D}(A), m\in \mathcal{P}_1(\mathbb{R}), \bo p\in H$, or
by
\begin{equation*}
  \mathcal H(t,\bo x,m,\bo p)=\langle\bo x,A^*\bo p\rangle+\langle\Gamma(m),\bo p\rangle+L^*(t,x_0,B^*\bo p),
\end{equation*}
for $t\in [0,T), \bo x\in H, m\in \mathcal{P}_1(\mathbb{R}), \bo p\in \mathcal{D}(A^*)$.
Thanks to Assumptions~\ref{hp},
$L^*$
%(and so the Hamiltonian $\mathcal H$)
is well-defined and finite.
%we know that the supremum in \eqref{supremo} is always attained. Hence, $L^*$ (and so the Hamiltonian $\mathcal H$) is a finite-valued function.

%Observe that the function $\mathcal H$ is well-defined just if either $\bo x\in\mathcal D(A)$ or $\bo p\in\mathcal D(A^*)$.
%Due to the \eqref{hjb}, this is equivalent to prove that the function $v$ satisfies
In general it is not true that $v_{\bo x}(t,\bo x)\in D(A^*)$ for all $(t,\bo x)\in[0,T)\times H$.
% This could be not true in general; i
 % In particular, this would force to impose conditions on the cost function $g$, to have at least $v_{\bo x}(T,\bo x)\in D(A^*)$ {\color{magenta}(PERCH\'E? la HJB vale per $t<T$)} {\color{blue} È vero, ma mi sembrerebbe strano avere $v_{\bo x}(t,\bo x)\in D(A^*)$ per ogni $t<T$ e poi $v_{\bo x}(T,\bo x)\in D(A^*)$, soprattutto in un quadro di soluzioni classiche (lì sarebbe proprio impossibile). Possiamo togliere il commento sulle condizioni per $g$, lasceremi comunque il resto}.
  A way to overcome this problem is to consider mild
  solutions to~\eqref{hjb}, as we will do in the next section.
  % at least in a particular case, finding a mild solution for the problem \eqref{hjb}.

  The optimal control for the player is given in feedback by computing
  when the supremum in~\eqref{Lstar}
  is attained.
% {\color{magenta}QUI DI SEGUITO SI STA ASSUMENDO $U=\mathbb{R}$} {\color{blue}  Sì, aggiungiamolo. Possiamo pure formulare tutto dall'inizio con $U=\R$.}
  Differentiating with respect to $u$, we have to impose the condition
$$
B^*\bo p-L_u(t,x_0,u)=0\,.
$$
Since $L$ is strictly convex in $u$ for Assumption \ref{hp}, $L_u$ is strictly increasing (hence invertible), and the equation can be solved with respect to $u$, giving
$$
u=u(t,x_0,B^*p)=L_u^{-1}(t,x_0,B^*\bo p)=L^*_u(t,x_0,B^*\bo p).
$$
where the last equality comes from the property of the convex conjugate, see e.g.\ \cite{Ekeland}. Hence, the optimal control is given in feedback form by
\begin{equation}\label{optctrl}
u=L^*_u\big(t,X_0(t),B^*v_{\bo x}(t,\bo X(t))\big)\,.
\end{equation}
Observe that $L^*$ is strictly convex, with superlinear growth and differentiable in the control variable $u$, see e.g. \cite{Ekeland,GomesRicciardi}.

%{\color{magenta} MA CHE REGOLARIT\`A HA $L^*$?} {\color{blue} Convessa e a crescita superlineare. Abbiamo anche la stretta convessità e la differenziabilità se $L$ è differenziabile in $u$ (ipotesi che ho aggiunto). Sono risultati di Ekeland-Temam, che abbiamo già citato, più qualche aggiunta che feci io in un vecchio lavoro (che ho messo nelle cit).}

\subsection{The Fokker-Planck equation and the Mean Field Game system}\label{sec_FP}
Now we turn out to the equation satisfied by the law of the optimal trajectory associated to the controll $u$. Assuming that the process starts from an initial configuration $\bo X(0)$ with law $\overline m$,
% {\color{magenta}(ma al tempo $0$ $\bo X$ \`e un punto)} {\color{blue} questo solo quando si considera la \eqref{SDE}, che serve per definire la value function. Le traiettorie che partono dal tempo $t=0$ evolvono da una generica legge iniziale},
from \eqref{optctrl}, the dynamics of the optimal trajectory
is formally given by
 the following SDE
 \begin{equation}\label{Xoptimal}
   \begin{cases}
d\bo X(t)=\big[A\bo X(t)+\Gamma(\tilde m_0(t))+BL^*_u\big(t,X_0(t),B^*v_{\bo x}(t,\bo X(t))\big)\big]dt+\Sigma\, dW(t)\,,\\
\mathscr L(\bo X(0))=\overline{m}.
\end{cases}
\end{equation}
% {\color{magenta}PERCH\'E C'\`E $\bo x$ dentro $\overline m$?} {\color{blue} Come sopra, ora $\overline m$ è una misura generica, non un punto}
By \eqref{supremo} the SDE can be rewritten as
$$
\begin{cases}
d\bo X(t)=\mathcal H_{\bo p}\big(t,\bo X(t),\tilde m_0(t),v_{\bo x}(t,\bo X(t))\big)\,dt+\Sigma\, dW(t)\,,\\
\mathscr L(\bo X(0))=\overline{m}.
\end{cases}
$$
If we define $m(t)=\mathscr L(X(t))$ for $t\in[0,T]$,
then $m$ satisfies the Fokker-Planck equation
\begin{equation}\label{fp}
\begin{cases}
\ds m_t-\frac12\tr\big(\Sigma\Sigma^*m_{\bo x\bo x}\big)+\div\Big(m\mathcal H_{\bo p}\big(t,\bo x,\tilde m_0(t),v_{\bo x}\big)\Big)=0\,,\\
\ds m(0)=\overline{m}\,.
\end{cases}
\end{equation}
Since $m$ is a measure, the previous equation
should be interpreted in the sense of distributions,
as specified by the following definition.
\begin{definition}
  Let $m\colon[0,T]\to\mathcal P_1(H)$. Then $m$ is a distributional solution of the Fokker-Planck equation \eqref{fp} if $m(0)=\overline{m}$ and, for all $\phi\in C^{\infty}_b([0,T]\times H)$
  % {\color{magenta}QUI SERVE QUALCHE CONDIZIONE DI CRESCITA; IO METTEREI DERIVATE BOUNDED} {\color{blue} Giusto, $C^{\infty}_b$ è la scelta migliore}
  and all $t\in(0,T]$, we have
\begin{gather}
\int_H\phi(t,\bo x)m(t,d\bo x)=\int_H\phi(0,\bo x)\,\overline{m}(d\bo x)\label{testfp}
\\+\int_0^t\int_H\Big(\phi_s(s,\bo x)+\frac12\mathrm{tr}\big(\Sigma\Sigma^*\phi_{\bo x\bo x}(s,\bo x)+\langle\mathcal H_{\bo p}\big(s,\bo x,\tilde m_0(s),v_{\bo x}(s,\bo x)\big),\phi_{\bo x}(s,\bo x)\rangle\Big)m(s,d\bo x)ds\,.\notag
\end{gather}
\end{definition}
Observe that the operator $\mathcal H_p$ involves the quantity $A\bo x$. Hence, one should impose the additional assumption
%{\color{magenta}(ho messo cos\`i, perch\'e il supp di una misura per definziione \`e chiuso)}
$m(t)(\mathcal{D(A)})=1$ for a.e.\ $t\in[0,T]$.
% $\supp(m)\subseteq[0,T]\times\mathcal D(A)$ {\color{magenta}}.
Alternatively, the test functions $\phi$ should satisfy the condition $\phi_{\bo x}(t,\bo x)\in D(A^*)$
for all $ (t,\bo x)\in(0,T)\times H$.
%{\color{magenta} perch\'e supp? $\varphi_x\in D(A^*)$}. {\color{blue} Sì, ho corretto}

In Mean Field Games, players try to predict the behavior of the population and then optimize their own cost based on this prediction.  In a MFG equilibrium, each player's strategy coincides with their prediction of the mean field. When the players' strategies align with their predictions of the mean field, and the mean field evolves according to the aggregate of these strategies, you reach an equilibrium state where no player has an incentive to unilaterally deviate from their chosen strategy.

In formula, this means that,
denoting by $m_0$ the law of the first marginal $\bo X_0$ of the optimal process $\bo X$ with dynamics
\eqref{Xoptimal}, we should have $m_0=\tilde m_0$. This gives us a system of PDEs, where the backward Hamilton-Jacobi-Bellman equation 
for the value function $\mathcal{V}$
is coupled with the forward Fokker-Planck equation for the density of the population $m$. The system is the following one:
\begin{equation}\label{mfg}
  \begin{cases}
\ds v_t+\frac12\tr\big(\Sigma\Sigma^*v_{\bo x\bo x}\big)+\mathcal H\big(t,\bo x,m_0(t),v_{\bo x}\big)+f\big(t,x_0,m_0(t)\big)=r v\,,\\
\ds m_t-\frac12\tr\big(\Sigma\Sigma^*m_{\bo x\bo x}\big)+\div\Big(m\mathcal H_{\bo p}\big(t,\bo x, m_0(t),v_{\bo x}\big)\Big)=0\,,\\
\ds m(0)=\overline{m}\,,\qquad v(T,\bo x)=g(x_0, m_0(T))\,.
\end{cases}
\end{equation}

\section{Solution in a particular case}\label{sec_solution}
In this section, we consider a particular case of the Mean Field Game system, with a specific choice of the coefficients, letting us
% . In particular, we want to show that we can consider a simpler system of PDEs instead of \eqref{mfg}, which gives us
to find
an explicit solution for the value function and the optimal control.

First of all, notice
that, due to the definition of $\Sigma$,
 we have
$$
\frac12\tr\big(\Sigma\Sigma^*v_{\bo x\bo x}(t,\bo x)\big)=\frac{\sigma^2}2v_{x_0x_0}(t,\bo x)\,,
$$
and the same occurs with the function $m$.

%We set $U=\mathbb{R}$. {\color{blue} Lo spostiamo prima?}
We then choose
the
Lagrangian cost, the running and terminal utility, and the $\gamma$ function, by setting
\begin{equation}\label{data1}
\begin{split}
&L(t,x_0,u)=\alpha u+\beta u^2\,,\qquad f(t,x_0,m_0)=f_0\left(t,x_0,\int_{\R} x\,m_0(dx)\right)\,,\\
&g(x_0,m_0)=g_0\left(x_0,\int_{\R}x\,m_0(dx)\right)\,,\qquad \gamma(m_0)=\gamma_0\int_{\R}x\,m_0(dx)\,,
\end{split}
\end{equation}
with $\alpha,\gamma_0\in\R$, $\beta>0$, $f_0\colon[0,T]\times\R^2\to\R$, $g_0\colon\R^2\to\R$. Hence, the utility functions $f$ and $g$ depends on the measure $m$ just as a function of the first marginal's average.

We define $\bo\mu(t)=(\mu_0(t),\mu_1(t))$ as the average of the process $\bo X(t)=(\bo X_0(t),\bo X_1(t))$, i.e., 
$$
\bo\mu(t)\coloneqq \E[\bo X(t)]=\int_H\bo x\,dm(x)\,.
%\mu_0(t)=\E[X_0(t)]=\int_H x_0\,dm(\bo x)=\int_\R x\,m_0(dx)\,.
$$
We assume that the functions $f_0$ and $g_0$ are linear with respect to both variables $(x_0,\mu_0)$, and independent of time:
\begin{equation}\label{data2}
f_0(t,x_0,\mu_0)=c(t)+\sigma_0(t) x_0+\theta(t) \mu_0\,,\qquad g_0(x_0,\mu_0)=c_T+\sigma_T x_0+\theta_T \mu_0\,.
\end{equation}
With this assumptions, we can  compute the optimal control in feedback form and the Hamiltonian of the system. We have
$$
L_u(t,x_0,u)=\alpha+2\beta u\,,\qquad L_u^{-1}(t,x_0,u)=\frac{u-\alpha}{2\beta}\,,\qquad L_u^{-1}(t,x_0,B^*\bo p)=\frac{B^*\bo p-\alpha}{2\beta}\,.
$$
Then the Hamiltonian $\mathcal H$ is
$$
\mathcal H(t,\bo x,m_0,\bo p)=\langle A\bo x,\bo p\rangle +\gamma_0p_0\int_{\R}x\,m_0(dx)+\frac{\big(B^*\bo p-\alpha\big)^2}{4  \beta}\,.
$$
We consider the function $\mathcal H^0:[0,T]\times H\times\R\times H$ defined by
$$
\mathcal H^0(t,\bo x,\mu_0,\bo p)=\langle A\bo x+\mu_0\bo\gamma,\bo p\rangle+\frac{\big(B^*\bo p-\alpha\big)^2}{4  \beta}=\langle A\bo x+\mu_0\bo\gamma,\bo p\rangle+\frac{\big(\langle \bo b,\bo p\rangle-\alpha\big)^2}{4  \beta}\,,
$$
with $\bo\gamma=(\gamma_0,\gamma_1)$. We have
$$
\mathcal H(t,\bo x,m_0,\bo p)=\mathcal H^0\left(t,\bo x,\int_\R x\,m_0(dx),\bo p\right),\quad H^0_{\bo p}(t,\bo x,\mu_0,\bo p)= A\bo x+\mu_0\bo\gamma+\frac{\big(\langle \bo b,\bo p\rangle-\alpha\big)\bo b}{2  \beta}.
$$
The Hamilton-Jacobi-Bellman equation has the form
$$
\begin{cases}
\ds v_t(t,\bo x)+\frac{\sigma^2}2v_{x_0x_0}(t,\bo x)+\langle A\bo x+\mu_0(t)\bo \gamma,v_{\bo x}(t,\bo x)\rangle+\frac{\big(\langle \bo b,v_{\bo x}(t,\bo x)\rangle-\alpha\big)^2}{4  \beta}\\\hspace{6.55cm}-rv(t,\bo x)+c(t)+\langle\bo\sigma(t),\bo x\rangle+\theta(t)\mu_0(t)=0\,,\\
\ds v(T,\bo x)=c_T+\langle\bo\sigma_T,\bo x\rangle+\theta_T\mu_0(T)\,,
\end{cases}
$$
where $\bo\sigma(t)=(\sigma_0(t),0)$ and $\bo\sigma_T=(\sigma_T,0)$. Observe that it is not necessary to compute the density of the population $m$ to find the value function $v$, since $v$ just depends on $\mu_0(\cdot)$. Hence, instead of the Fokker-Planck equation, we couple the Hamilton-Jacobi-Bellman equation with the one satisfied by $\mu_0(\cdot)$.

To do so, it suffices to take the equation \eqref{SDE} satisfied by $\bo X(t)$ and computing the average. Observe that we can write
$$
\mu_0(t)\bo\gamma=\Gamma_0\bo\mu(t)\,,\qquad\mbox{where}\quad \Gamma_0\colon H\to H\,,\quad \Gamma_0\bo\mu\coloneqq (\gamma_0 x_0,0)\,.
$$
Hence, we get
$$
\frac{d}{dt}\E[\bo X(t)]=(A+\Gamma_0)\E[\bo X(t)]+\bo\psi(t)\,,\qquad\mbox{where}\quad \bo\psi(t)\coloneqq \frac{\langle \bo b,\E\big[\bo v_{\bo x}(t,\bo X(t))\big]\rangle-\alpha}{2  \beta}\,\bo b\,.
$$
Calling $\bo{\bar\mu}=(\bar\mu_0,\bar\mu_1)=\E[\bo{\bar X}(0)]$, we obtain the equation for $\bo\mu$:
$$
\begin{cases}
\bo\mu'(t)=(A+\Gamma_0)\bo\mu(t)+\bo\psi(t)\,,\\
\bo\mu(0)=\bo{\bar\mu}\,.
\end{cases}
$$
\begin{comment}
Observe that, for $\bo x\in\mathcal D(A)$, $(A+\Gamma)\bo x=((a+\gamma_0)x_0+x_1(0),-x'_1(\cdot))$. Hence, it has the same structure of the operator $A$, with $a$ replaced by $a+\gamma_0$. This implies that we get the same formulas as \eqref{semigruppi} for the $C_0$-semigroup $\{e^{t(A+\Gamma)}\}_{t\in\R^+}$:
$$
e^{t(A+\Gamma)}\bo x=\left(e^{(a+\gamma_0)t}x_0+\intz\indic_{[-t,0]}(s)e^{(a+\gamma_0)(t+s)}x_1(s)\,ds,\,x_1(\cdot-t)\indic_{[-d+t,0]}(\cdot)\right)\,.
$$
Hence, we get
$$
\bo\mu(t)=\E[\bo X(t)]=e^{(A+\Gamma)t}\bo{\bar\mu}+\int_0^t e^{(A+\Gamma)(t-s)}\bo\psi(s)\,ds\,.
$$
\begin{align*}
\mu_0(t)
&=\bar\mu_0+\int_0^t\int_H\langle\mathcal H^0_p\big(s,\bo x,\mu_0(s),v_{\bo x}(s,\bo x),(1,0)\big)m(s,d\bo x)\,ds\\
&=\bar\mu_0+\int_0^t\int_H\langle A\bo x+\mu_0(s)\bo\gamma+\frac{\big(\langle \bo b,\bo v_{\bo x}(s,\bo x)\rangle-\alpha\big)\bo b}{2  \beta},(1,0)\rangle m(s,d\bo x)\,ds\\
&=\bar\mu_0+\int_0^t\int_H \Big(ax_0+\mu_0(s)\gamma_0+\frac{\big(\langle \bo b,\bo v_{\bo x}(s,\bo x)\rangle-\alpha\big)b_0}{2  \beta}\Big)m(s,d\bo x)\,ds\\
&=\bar\mu_0+\int_0^t\left[(a+\gamma_0)\mu_0(s)+\frac{b_0}{2  \beta}\left(\int_H \langle \bo b,\bo v_{\bo x}(s,\bo x)\rangle m(s,d\bo x)-\alpha\right)\right]ds\,.\\
\end{align*}
Hence, called $\psi:[0,T]\to\R$ the function defined by
$$
\psi(t)=\frac{b_0}{2  \beta}\left(\int_H \langle \bo b,\bo v_{\bo x}(t,\bo x)\rangle\, m(t,d\bo x)-\alpha\right)\,,
$$
we can write the equation satisfied by $\mu_0$:
$$
\begin{cases}
\mu_0'(t)=(a+\gamma_0)\mu_0(t)+\psi(t)\,,\\
\mu_0(0)=\bar\mu_0\,.
\end{cases}
$$
\end{comment}
Hence, we can write a simplified Mean Field Game system for the couple $(v,\bo\mu)$:
\begin{equation}\label{mfgesempio}
\begin{dcases}
\ds v_t(t,\bo x)+\frac{\sigma^2}2v_{x_0x_0}(t,\bo x)+\langle A\bo x+\mu_0(t)\bo \gamma,v_{\bo x}(t,\bo x)\rangle+\frac{\big(\langle \bo b,v_{\bo x}(t,\bo x)\rangle-\alpha\big)^2}{4  \beta}\\\hspace{6.55cm}-rv(t,\bo x)+c(t)+\langle\bo\sigma(t),\bo x\rangle+\theta(t)\mu_0(t)=0\,,\\[5pt]
\bo\mu'(t)=(A+\Gamma_0)\bo\mu(t)+\bo\psi(t)\,,\\[5pt]
\bo\mu(0)=\bo{\bar\mu}\,.\,,\qquad  v(T,\bo x)=c_T+\langle\bo\sigma_T,\bo x\rangle+\theta_T\mu_0(T)\,.
\end{dcases}
\end{equation}
The problem here lies in the function $\bar\psi$, since it still depends on the density function $m$. We show now that, in this particular case, the dependence on $m$ disappears, making the system \eqref{mfgesempio} solvable.

We look for a solution of the HJB equation of the form
$$
v(t,\bo x)=\langle \bo h(t),\bo x\rangle + k(t)\,,\qquad \bo h(t)=\big(h_0(t),h_1(t,\cdot)\big)\in H\,.
$$
Computing the derivatives, we find
$$
v_t(t,\bo x)=\langle \bo h'(t),\bo x\rangle+k'(t)\,,\qquad v_{\bo x}(t,\bo x)=\bo h(t)\,,\qquad v_{\bo x\bo x}(t,\bo x)=0\,.
$$
Hence, the Hamilton-Jacobi-Bellman equation becomes
$$
\begin{cases}
\ds\langle \bo h'(t),\bo x\rangle+k'(t)+\langle A^*\bo h(t),\bo x\rangle +\mu_0(t)\langle \bo h(t),\bo\gamma\rangle-r\langle\bo h(t),\bo x\rangle-rk(t)+\frac{\big(\langle \bo b,\bo h(t)\rangle-\alpha\big)^2}{4  \beta}\\
\ds\hspace{7.25cm}+c(t)+\langle\bo\sigma(t),\bo x\rangle+\theta(t) \mu_0(t)=0\,,\\
\ds\langle \bo h(T),\bo x\rangle+k(T)=c_T+\langle\sigma_T,\bo x\rangle+\theta_T\mu_0(T)\,.
\end{cases}
$$
This equation can be divided into two different equations, by considering first the terms containing $\bo x$ and then all the other ones. Moreover, observe that the function $\bo\psi$ is now given by
$$
\bo\psi(t)=\frac{\langle \bo b,\E\big[\bo h(t)\big]\rangle-\alpha}{2  \beta}\,\bo b=\frac{\langle \bo b,\bo h(t)\rangle-\alpha}{2  \beta}\,\bo b\,,
$$
which, as told before, does not depend on $m$.

Hence, the system \eqref{mfgesempio} is equivalent to the following system of three equations for the triple $(\bo h,k,\bo\mu)$:
\begin{equation}\label{equiva}
\begin{cases}
\ds  \bo h'(t)+(A^*-r)\bo h(t)+\bo\sigma(t)=0\,,\\[5pt]
\ds k'(t)-rk(t)+\mu_0(t)\langle \bo h(t),\bo\gamma\rangle+\frac{\big(\langle \bo b,\bo h(t)\rangle-\alpha\big)^2}{4  \beta}+c(t)+\theta(t) \mu_0(t)=0\,,\\[5pt]
\ds \bo\mu'(t)=(A+\Gamma_0)\bo\mu(t)+\frac{\langle \bo b,\bo h(t)\rangle-\alpha}{2  \beta}\,\bo b\,,\\[5pt]
\bo\mu(0)=\bo{\bar\mu}\,,\qquad \bo h(T)=\bo\sigma_T\,,\qquad k(T)=c_T+\theta_T\mu_0(T)\,.
\end{cases}
\end{equation}
Observe that we cannot
expect to have
% {\color{magenta}(cannot nel senso che proprio si dimostra che non pu\`o essere? oppure che in generale non \`e vero? )} {\color{blue} Penso che non possa proprio essere, ma possiamo anche scrivere "we cannot expect to have"} have
$\bo h(t)\in D(A^*)$ for all $t\in[0,T)$.
% (it is immediate for $t=T$ {\color{magenta}(ma le dinamiche valgono per $t<T$)} {\color{blue} Stesso commento di prima. Tra l'altro la stessa osservazione sul tempo $t=T$ era presente nel vostro lavoro con F. Masiero}).
Hence, the equation for $\bo h$ has to be intended in a mild sense.
Since the equation for $\bo h$ is decoupled (it does not depend on $\mu_0$ or $k$), the mild solution is:
$$
\bo h(t)=e^{(T-t)(A^*-r)}\bo\sigma_T+\int_t^T e^{(s-t)(A^*-r)}\bo\sigma(s)\,ds\,.
$$
From here we can solve the equation for $\bo\mu$. Observe that, for $\bo x\in\mathcal D(A)$, $(A+\Gamma)\bo x=((a+\gamma_0)x_0+x_1(0),-x'_1(\cdot))$. Hence, $A+\Gamma$ has the same structure of the operator $A$, if we replace $a$ with $a+\gamma_0$. This implies that we get the same formulas as \eqref{semigruppi} for the $C_0$-semigroup $\{e^{t(A+\Gamma)}\}_{t\in\R^+}$:
$$
e^{t(A+\Gamma)}\bo x=\left(e^{(a+\gamma_0)t}x_0+\intz\indic_{[-t,0]}(s)e^{(a+\gamma_0)(t+s)}x_1(s)\,ds,\,x_1(\cdot-t)\indic_{[-d+t,0]}(\cdot)\right)\,.
$$
Hence, we obtain
$$
\bo\mu(t)=e^{(A+\Gamma)t}\bo{\bar\mu}+\int_0^t e^{(A+\Gamma)(t-s)}\bps{s}\,ds\,.
$$
%$$
%\mu_0(t)=e^{(a+\gamma_0)t}\bar\mu_0+\frac{b_0}{2  \beta}\int_0^t e^{(a+\gamma_0)s}\big(\langle \bo b,\bo h(s)\rangle-\alpha\big)\,ds\,,
%$$
Finally, the equation for $k$ gives
$$
k(t)=e^{r(t-T)}(c_T+\theta_T\mu_0(T))+\int_t^T e^{r(t-s)}\Big(\mu_0(s)\big(\langle \bo h(s),\bo\gamma\rangle+\theta(s)\big)+\frac{\big(\langle \bo b,\bo h(s)\rangle-\alpha\big)^2}{4  \beta}+c(s)\Big)ds\,.
$$
Hence, we have solved the system \eqref{equiva} in a mild sense, finding the solution
\begin{equation}\label{khm}
\begin{cases}
\ds \bo h (t)&\ds=e^{(T-t)(A^*-r)}\bo\sigma_T+\int_t^T e^{(s-t)(A^*-r)}\bo\sigma(s)\,ds\,,\\
\ds \bo\mu (t)&\ds=e^{(A+\Gamma)t}\bo{\bar\mu}+\int_0^t e^{(A+\Gamma)(t-s)}\frac{\langle \bo b,\bo h (s)\rangle-\alpha}{2  \beta}\,\bo b\,ds\,,\\
\ds k (t)&\ds=e^{r(t-T)}(c_T+\theta_T\mu_0 (T))\\
&\ds+\int_t^T e^{r(t-s)}\Big(\mu_0 (s)\big(\langle \bo h (s),\bo\gamma\rangle+\theta(s)\big)+\frac{\big(\langle \bo b,\bo h (s)\rangle-\alpha\big)^2}{4  \beta}+c(s)\Big)ds\,.
\end{cases}
\end{equation}

Since $\bo h\notin C^1([0,T];H)$, the equivalence between \eqref{mfgesempio} and \eqref{equiva} is formal, and we cannot say rigorously that the function $\bo v$ obtained is actually the value function of our problem % {\color{magenta}(beh ma anche se i due sistemi fossero equivalente servirebbe una verifica, mi sembra)} \textcolor{blue}{Sì}.
We handle this problem in the next subsection. 

\subsection{A verification theorem}\label{sec_verification}
Here we want to prove that $v$ is the value function $\mathcal{V}$
defined in \eqref{value}.

To do this, we will make use of the Yosida approximation $A_n$ of $A$ (as in \cite{GozziMasiero2}).
It is well-known that
%, under the results of the previous Lemma, we have
\begin{equation}\label{conv}
e^{tA_n}\bo x\to e^{tA}\bo x\,,\qquad e^{tA_n^*}\bo x\to e^{tA^*}\bo x\qquad \forall x\in H\,,
\end{equation}
where the convergence is uniform for $t\in[0,T]$ and locally uniform for $\bo x\in H$.

For $n\in \mathbb{N}$ large enough, we  consider, for $u\in\mathcal U$, the unique strong solution of the SDE
\begin{equation}\label{SDEn}
\begin{cases}
d\bo X^n(s)=\big[A_n\bo X^n(s)+\gamma_0\mu_0^n(s)+Bu(s)\big]ds+\Sigma dW(s)\,,\\
\bo X^n(t)=\bo x\,,\\
u(t+\xi)=\delta(\xi)\,,\qquad\forall\,\xi\in[-d,0]\,,
\end{cases}
\end{equation}
associated to the same objective functional \eqref{obj} (with $\bo X^{t,\bo x,u}$ replaced by $\bo X^n$). Here $\mu_0^n$ represents the law of $\bo X^n_0$ if $u$ is the optimal control, i.e., in a MFG equilibrium.

The simplified MFG system for $(v^n,\bo\mu^n)$ associated to this approximated problem is given by
\begin{equation}\label{mfgn}
\begin{cases}
\ds v^n_t(t,\bo x)+\frac{\sigma^2}2v^n_{x_0x_0}(t,\bo x)+\langle A^n\bo x+\mu_0^n(t)\bo \gamma,v^n_{\bo x}(t,\bo x)\rangle+\frac{\big(\langle \bo b,v^n_{\bo x}(t,\bo x)\rangle-\alpha\big)^2}{4  \beta}\\\hspace{6.55cm}-rv^n(t,\bo x)+c(t)+\langle\bo\sigma(t),\bo x\rangle+\theta(t)\mu^n_0(t)=0\,,\\
\ds (\bo\mu^n)'(t)=(A+\Gamma_0)\bo\mu^n(t)+\bo\psi^n(t)\,,\\
\ds\bo\mu^n(0)=\bo{\bar\mu}\,,\qquad  v^n(T,\bo x)=c_T+\langle\bo\sigma_T,\bo x\rangle+\theta_T\mu_0^n(T)\,.
\end{cases}
\end{equation}
As before, the solution is given by
\begin{equation}\label{defvn}
v^n(t,\bo x)=\langle\bo h^n(t),\bo x\rangle+k^n(t)\,,
\end{equation}
where the triple $(\bo h^n,k^n,\bo\mu^n)$ solve the system
\begin{equation}\label{equivan}
\begin{cases}
\ds  \bo (h^n)'(t)+(A_n^*-r)\bo h^n(t)+\bo\sigma(t)=0\,,\\
\ds (k^n)'(t)-rk^n(t)+\mu_0^n(t)\langle \bo h^n(t),\bo\gamma\rangle+\frac{\big(\langle \bo b,\bo h^n(t)\rangle-\alpha\big)^2}{4  \beta}+c(t)+\theta(t) \mu_0^n(t)=0\,,\\
\ds (\bo\mu^n)'(t)=(A_n+\Gamma_0)\bo\mu(t)+\frac{\langle \bo b,\bo h^n(t)\rangle-\alpha}{2  \beta}\,\bo b\,,\\
\bo\mu^n(0)=\bo{\bar\mu}\,,\qquad \bo h^n(T)=\bo\sigma_T\,,\qquad k^n(T)=c_T+\theta_T\mu_0^n(T)\,.
\end{cases}
\end{equation}
As before, the solution is given by
$$
\begin{cases}
\ds \bo h^n(t)&\ds=e^{(T-t)(A_n^*-r)}\bo\sigma_T+\int_t^T e^{(s-t)(A_n^*-r)}\bo\sigma(s)\,ds\,,\\
\ds \bo\mu^n(t)&\ds=e^{(A_n+\Gamma)t}\bo{\bar\mu}+\int_0^t e^{(A_n+\Gamma)(t-s)}\frac{\langle \bo b,\bo h^n(s)\rangle-\alpha}{2  \beta}\,\bo b\,ds\,,\\
\ds k^n(t)&\ds=e^{r(t-T)}(c_T+\theta_T\mu_0^n(T))\\
&\ds+\int_t^T e^{r(t-s)}\Big(\mu_0^n(s)\big(\langle \bo h^n(s),\bo\gamma\rangle+\theta(s)\big)+\frac{\big(\langle \bo b,\bo h^n(s)\rangle-\alpha\big)^2}{4  \beta}+c(s)\Big)ds\,.
\end{cases}
$$

Since $A_n$ and $A^*_n$ are bounded,
% and continuous linear operator from $H$ to $H$, 
for all $\bo x\in H$ we have $e^{tA_n}\bo x$ and $e^{tA_n^*}$ differentiable in $t$,
% with
% $$
% \frac{d}{dt}e^{tA_n}\bo x=A_n e^{tA_n}\bo x\,,\qquad \frac{d}{dt}e^{tA^*_n}\bo x=A^*_n e^{tA^*_n}\bo x\,.
% $$
and
then $\bo h^n\in C^1([0,T];H)$, $\mu^n\in C^1([0,T];H)$ and $k^n\in C^1([0,T];\R)$, and so they are classical solutions of \eqref{equivan}. Hence, the formal computations given previously now become rigorous, and we know that $(v^n,\bo\mu^n)$, with $v^n$ defined in \eqref{defvn}, is a classical solution of \eqref{mfgn}.\\

Now we are ready to conclude with the following verification theorem for the value function $v$.
\begin{theorem}
Consider the process $\bo X$ satisfying \eqref{SDE}, with $t=0$ and with $\mathcal L(\bo X(0))=\bar m$. Let $\bo{\bar\mu}=\E[\bo X(0)]$ and let $\mathcal J$, $\mathcal{V}$ be defined as in \eqref{obj}, \eqref{value}.

If the data $f,g,L,\gamma$ are defined as in \eqref{data1}, \eqref{data2}, with continuous functions $\sigma_0, c,\theta$, then we have
$$
V(t,\bo x)=v(t,\bo x)=\langle \bo h(t),\bo x\rangle+k(t)\,,
$$
with $\bo h$, $k$ defined in \eqref{khm}.
\end{theorem}
\begin{proof}
We start considering, for a generic control $u(\cdot)\in\mathcal U$, the process $\bo X^n(\cdot)$ solving \eqref{SDEn}. For $\bo v^n$ defined as in \eqref{defvn}, and solution of \eqref{mfgn}, we apply Ito's formula to the process $\{e^{-r(s-t)}v^n(s,\bo X^n(s))\}_{s\ge t}$, obtaining
\begin{align*}
&\E\left[e^{-r(T-t)}v^n(T,\bo X^n(s))\right]=\E[v^n(t,\bo x)]\\
+&\E\left[\int_t^T e^{-r(s-t)}\left(-rv^n(t,\bo X^n(s))+v^n_t(t,\bo X^n(s))+\frac{\sigma^2}2v^n_{x_0x_0}(t,\bo X^n(s))\right)\,ds\right]\\
+&\E\left[\int_t^T e^{-r(s-t)}\big\langle A_n\bo X^n(s)+Bu(s)+\bo\gamma \mu_0^n(s),v^n_{\bo x}(s,\bo X^n(s))\big\rangle\,ds\right]\,.
\end{align*}
Since $v^n$ satisfies \eqref{mfgn}, the above equality becomes
\begin{equation}\label{step1n}
\begin{split}
&\hspace{3cm}e^{-r(T-t)}\Big(c_T+\theta_T\mu^n_0(T)+\E[\langle \bo\sigma_T,\bo X^n(T)\rangle]\Big)=v^n(t,\bo x)\\
&+\,\E\left[\int_t^T e^{-r(s-t)}\left(u(s)\big\langle \bo b, \bo h^n(s)\big\rangle-\frac{\big(\langle \bo b,\bo h^n(s)\rangle-\alpha\big)^2}{4  \beta}-c(s)-\langle\bo\sigma(s),\bo X^n(s)\rangle-\theta\mu^n_0(s)\right)ds\right].
\end{split}
\end{equation}
Observe that, thanks to \eqref{conv}, we have $(\bo h^n( t),k^n( t),\bo\mu^n( t))\to(\bo h( t),k( t),\bo\mu( t))$ uniformly in $t\in[0,T]$. This implies $v^n(t,\bo x)\to v(t,\bo x)$ uniformly in $t\in[0,T]$ and locally uniformly in $\bo x\in H$.

Moreover, since $\bo X^n(\cdot)$ is a strong solution of \eqref{SDEn}, it is also a mild solution.
 Hence, we can write
$$
\bo X^n(s)=e^{(s-t)A_n}\bo x+\int_t^s e^{(s-r)A_n}\big[\gamma_0\mu^n_0(r)+Bu(r)\big]\,dr+\int_t^s e^{(s-r)A_n}\Sigma\,dW(r)\,,\quad \forall s\in[t,T]\,.
$$
Again from \eqref{conv}, we have
$$
\bo X^n\to \bo X\qquad\mbox{in }L^2_{\mathcal P}(\Omega\times[0,T];H)\,.
$$
Hence, we can pass to the limit in \eqref{step1n}, obtaining
\begin{equation}\label{step1}
\begin{split}
v  (t,\bo x)&+\E\left[\int_t^T e^{-r(s-t)}\left(u(s)\big\langle \bo b, \bo h  (s)\big\rangle-\frac{\big(\langle \bo b,\bo h  (s)\rangle-\alpha\big)^2}{4  \beta}\right)ds\right]\\
&=\E\left[\int_t^T e^{-r(s-t)}\Big(c(s)+\langle\bo\sigma(s),\bo X  (s)\rangle+\theta\mu _0(s)\Big)\,ds\right]\\
&+\E\Bigg[e^{-r(T-t)}\Big(c_T+\theta_T\mu  _0(T)+\langle \bo\sigma_T,\bo X  (T)\rangle\Big)\Bigg].
\end{split}
\end{equation}
Observe that the last two lines are equal to
$$
\mathcal J(t,\bo x,u)+\int_t^T e^{-r(s-t)}\left[\alpha u(s)+\beta u(s)^2\right]\,ds
$$
This implies
\begin{align*}
v(t,\bo x)&=\mathcal J(t,\bo x,u)+\int_t^T e^{-r(s-t)}\left(u(s)\big(\alpha-\big\langle \bo b, \bo h  (s)\big\rangle\big)+\beta u(s)^2+\frac{\big(\langle \bo b,\bo h  (s)\rangle-\alpha\big)^2}{4  \beta}\right)ds\\
&=\mathcal J(t,\bo x,u)+\int_t^T \frac{e^{-r(s-t)}}{4\beta}\Big(2\beta u(s)-\big(\langle\bo b,\bo h(s)\rangle-\alpha\big)\Big)^2ds
\end{align*}
Since the quantity inside the integral is positive, we have
$$
v(t,\bo x)\ge \mathcal J(t,\bo x, u)\qquad\forall u\in\mathcal U\implies v(t,\bo x)\ge \mathcal{V}(t,\bo x)\,.
$$
Moreover, choosing
$$
\bar u(s)=\frac{\langle\bo b,\bo h(s)\rangle-\alpha}{2\beta}\,,
$$
we get
$$
v(t,\bo x)=\mathcal J(t,\bo x,\bar u)\implies v(t,\bo x)\le \mathcal{V}(t,\bo x)\,.
$$
This implies $v(t,\bo x)=V(t,\bo x)$ and concludes the proof.

\end{proof}

% \newpage
\printbibliography

\end{document}

%% file: mymacro.tex
\newcommand{\tr}{\operatorname{tr}}

\renewcommand{\div}{\operatorname{div}}

\def\leq{\leqslant}
\def\geq{\geqslant}

\def \R{\mathbb{R}}
\def \E{\mathbb{E}}
\def \P{\mathbb{P}}

\newcommand{\leqnomode}{\tagsleft@true}
\newcommand{\reqnomode}{\tagsleft@false}

\newcommand{\eps}{\varepsilon}

\newcommand{\norm}[1]{\ensuremath{\left\Arrowvert #1 \right\Arrowvert}}

\newcommand{\indic}{\mathlarger{\mathbbm{1}}}

\newcommand{\bo}[1]{\boldsymbol{#1}}

\newcommand{\ds}{\displaystyle}

\newcommand{\intz}{\displaystyle\int_{-d}^0}
\newcommand{\bps}[1]{\frac{\langle \bo b,\bo h(#1)\rangle-\alpha}{2\beta^2}\,\bo b}